\documentclass[11pt,a4paper,twoside]{amsart}
\usepackage[T1]{fontenc}
\usepackage[utf8]{inputenc}

\usepackage{amsmath,amssymb,amsfonts,amsthm}
\usepackage{multirow}
\usepackage{hyperref}
\usepackage{graphicx}
\usepackage{MnSymbol}
\usepackage{xcolor}
\usepackage{enumerate}

\newcommand{\cg}{{\mathfrak g}}

\DeclareMathOperator{\id}{id}
\DeclareMathOperator{\Ss}{S}
\DeclareMathOperator{\ad}{ad}
\DeclareMathOperator{\vol}{\nu}

\def\to{\longrightarrow}

\def\eps{\epsilon}

\newcommand{\cD}{{\mathcal D}}

\def\cg{{\mathfrak g}}

\newcommand{\Sph}{\mathbb{S}}

\def\<{\langle}
\def\>{\rangle}

\def \A{{\mathcal{A}}}
\def \D{{\mathcal{D}}}

\def\R{\mathbb{R}}

\newtheorem{The}{Theorem}[section]
\newtheorem{prop}[The]{Proposition}
\newtheorem{example}[The]{Example}
\newtheorem{Lem}[The]{Lemma}
\newtheorem{cor}[The]{Corollary}
\theoremstyle{definition}
\newtheorem{Def}[The]{Definition}

\newtheorem{Rem}[The]{Remark}

\thanks{2010 Mathematical Subject Classification.
 58D05, 35Q35, 53C22, 53C80.}

\newcommand{\arxivonly}[1]{#1}
\newcommand{\paperonly}[1]{}

\begin{document}
\title[Geodesic interpretation of quasi-geostrophic equations]{Geodesic interpretation of the global quasi-geostrophic equations}
\author{Klas Modin and Ali Suri}
\address{Klas Modin, Department of Mathematical Sciences, Chalmers University of
Technology and University of Gothenburg, SE-412 96, Gothenburg,
Sweden}
\email{klas.modin@chalmers.se}
\address{Ali Suri, Universit\"{a}t Paderborn, Warburger Str. 100,
33098 Paderborn, Germany}
\email{asuri@math.upb.de}
\maketitle {\hspace{2.5cm}}

\begin{abstract}
We give an interpretation of the global shallow water quasi-geostrophic equations on the sphere $\Sph^2$ as a geodesic equation on the central extension of the quantomorphism group on $\Sph^3$. The study includes deriving the model as a geodesic equation for a weak Riemannian metric, demonstrating smooth dependence on the initial data, and establishing global-in-time existence and uniqueness of solutions.
We also prove that the Lamb parameter in the model has a stabilizing effect on the dynamics: if it is large enough, the sectional curvature along the trade-wind current is positive, implying conjugate points.

\textbf{Keywords}: quantomorphism group, global quasi-geostrophic equations, central extension, geodesic equations, conjugate points
\end{abstract}

\pagestyle{headings} 

\arxivonly{\tableofcontents}

\section{Introduction}

The quasi-geostrophic model for atmospheric and oceanographic flows was formulated by Charney in 1949~\cite{Ch1949}. 
Roughly, it is derived as follows.
On a domain where the horizontal length scale is much greater than the vertical, the 3-D Navier-Stokes equations are well approximated by the 2-D shallow-water equations.
In a rotating frame of reference, these equations contain the Coriolis force, whose strength compared to the inertia of a typical horizontal velocity gives rise to the dimensionless Rossby number.
In the limit of small Rossby numbers, the waves due to gravity become highly oscillatory and can be filtered out.
This sequence of approximations leads to the quasi-geostrophic equations.
They take the form of an incompressible, inviscid 2-D fluid, expressed in the potential vorticity function $q$ as
\begin{subequations}\label{eq:original_qgs}
\begin{equation} \label{eq:transporteq}
    \partial_tq + \{\psi,q\} = 0,
\end{equation}
with a relation between $q$ and the stream function $\psi$ to be discussed next.

Charney considered the quasi-geostrophic model on a strip located at an approximately constant latitude.
This so called ``$\beta$-plane approximation'' allows the Coriolis parameter to be varied linearly in space, which in turn leads to a relation between potential vorticity and stream function given by
\begin{equation} \label{eq:beta_approx}
    q = \Delta\psi + \beta z + h
\end{equation}
\end{subequations}
where $z$ denotes the (linearized) latitude and the function $h$ represents the bottom floor topography.

For a global fluid on the entire sphere $\Sph^2$, the tacit assumption has often been to take as quasi-geostrophic approximation the same equations~\eqref{eq:original_qgs} with $z=\cos\vartheta$ for the latitudinal angle $\vartheta \in [-\pi,\pi]$.
However, Charney's basic assumptions in the $\beta$-plane approximation then fails.
A correction was given by Schubert and Silvers~\cite{Schubert}, Verkley~\cite{Verkley}, and also via a different method by Luesink et~al.~\cite{Luesink-2024}, which led to the \emph{global quasi-geostrophic equations}
\begin{equation}\label{eq-SWQG}
      \partial_tq + \{\psi,q\} = 0, \qquad q = (\Delta - \gamma z^2)\psi + \frac{2z}{{\rm Ro}} + 2z h .
\end{equation}
Here, $\gamma$ is the dimensionless Lamb parameter defined as 
\[
\gamma = \frac{4\Omega^2 a^2}{g H} = \frac{r^2}{L^2{\rm Bu}},
\]
where $r$ is the radius of the sphere, $\Omega$ is the rotation frequency, $g$ is the gravitational acceleration, $H$ is the typical thickness of the fluid layer, and $L$ is a characteristic length scale. 
Moreover, the constant  $\mathrm{Bu}$ represents the Burger number, defined as $\mathrm{Bu} = \left(\frac{\mathrm{Ro}}{\mathrm{Fr}}\right)^2$, where $\mathrm{Ro}$ is the Rossby number, representing the ratio between the typical horizontal velocity and the velocity induced by the rotation of the sphere, while $\mathrm{Fr}$ is the Froude number, representing the ratio between the typical horizontal velocity and the speed of the fastest gravity wave. 
The Burger number quantifies the relative dominance of rotation and stratification in the fluid dynamics of the system.

In this paper we carry out a geometrical analysis of the global quasi-geostrophic equations~\eqref{eq-SWQG}, with the following results.

\begin{enumerate}
    \item In section~\ref{sec:geometry}, we show that the equations can be cast as geodesic equations on the infinite-dimensional group of quantomorphisms on~$\mathbb{S}^3$.
    This result is analogous to Arnold's~\cite{Arnold66} discovery that the incompressible Euler equations can be interpreted as geodesic equations on the group of volume-preserving diffeomorphisms.
    \item In section~\ref{sec:analysis}, we use the geodesic interpretation to give local and global well-posedness results.
    These results are analogous to Ebin and Marsden's~\cite{Ebin-Marsden} local analysis based on Arnold's discovery, and to Ebin's~\cite{Ebin1} global well-posedness result for symplectomorphism groups.
    \item In section~\ref{sec:curvature}, we study Misio{\l }ek's criterion for positivity on the sectional curvature and we show that the Lamb parameter $\gamma$ has a stabilizing effect on the dynamics, in the sense that if it is large enough the sectional curvature becomes positive.
\end{enumerate}

These results are directly linked to the work of Ebin and Preston~\cite{Ebin-Preston} and Lee and Preston~\cite{Lee-Pre} on the analysis of the equations \eqref{eq:original_qgs}.
Relative to their work, the main complication for the equations~\eqref{eq-SWQG} arises from the non-isotropy of the infinite-dimensional Riemannian metric due to the term $\gamma z^2$ in the relation between the stream function $\psi$ and the potential vorticity~$q$.

\medskip

\noindent\textbf{Acknowledgements}.
The first author was supported by the Swedish Research Council (grant number 2022-03453) and the Knut and Alice Wallenberg Foundation (grant numbers WAF2019.0201).
The second author was supported by the Deutsche Forschungsgemeinschaft (DFG, German Research Foundation), project 517512794.

\section{Geodesic interpretation} \label{sec:geometry}
\subsection{The Hopf fibration} 
We give here a summary of the Hopf fibration $\pi : \Sph^3 \to \Sph^2$, which is a nontrivial circle bundle over $\Sph^2$. 

A point on $\Sph^3$ can be represented by two complex numbers $w_1 $ and $w_2$ satisfying $|w_1 |^2 + |w_2|^2 = 1$. 
We shall use the Hopf coordinates, where $w_1 $ and $w_2$ are expressed as
\[
w_1 = \cos\eta e^{i\xi_1} , \quad w_2 = \sin\eta e^{i\xi_2} ,
\]
with $\eta \in (0, \pi/2)$ and $\xi_1, \xi_2 \in (0, 2\pi)$.
The Euclidean metric on $\mathbb{C}^2$ induces the standard metric on $\Sph^3$, given by
\begin{equation}\label{standard metric 1}
g^{\Sph^3} = \cos^2\eta \, d\xi_1^2 + \sin^2\eta \, d\xi_2^2 + d\eta^2.
\end{equation}
Set $\lambda = 2\eta$, $\zeta = \xi_1 - \xi_2$, and $\chi=\frac{1}{2}(\xi_1+\xi_2)$. 
Then the standard metric \eqref{standard metric 1} on $\Sph^3$ can be written as
\begin{equation}\label{standard metric 2}
g^{\Sph^3}=\frac{1}{4} \big( d\lambda^2 +\sin^2\lambda  d\zeta^2\Big) + \big(d\chi  - \frac{1}{2} \cos\lambda  d\zeta\big)^2  .
\end{equation}
Note that
\[
g^{\Sph^2} = \frac{1}{4} \left( d\lambda^2 + \sin^2\lambda \, d\zeta^2 \right),
\]
is the metric on $\Sph^2$ of radius $\frac{1}{2}$.

In this formulation, the projection $\pi:\Sph^3 \to \Sph^2$  maps the point $(w_1,w_2)= ( \cos\eta e^{i\xi_1} ,  \sin \eta e^{i\xi_2})$ to
\[
\big( \bar{w_1}w_2, \frac{1}{2}(|w_1|^2-|w_2|^2)\big) =  \frac{1}{2}  \big( \sin(\lambda  ) \cos(\xi_1-\xi_2) ,  \sin(\lambda) \sin(\xi_1-\xi_2)  ,\cos(\lambda)  \big).
\]
As we can see from \eqref{standard metric 1} and \eqref{standard metric 2}, the surjection $\pi$ is a Riemannian submersion.

The circle action of $\Sph^1$ on $\Sph^3$ is
\[
e^{it}(e^{i\xi_1} \cos\eta, e^{i\xi_2} \sin \eta)=(e^{i(\xi_1+t)} \cos\eta, e^{i(\xi_2+t)} \sin \eta).
\]
Indeed, this action generates the fibers of the Hopf fibration, and the corresponding Reeb field is
\[
\partial_{\xi_1} +\partial_{\xi_2}.
\]
This is a Killing vector field for the metric \eqref{standard metric 1}, as it preserves the Riemannian metric on $\Sph^3$ and is equivariant under the symmetry of the $\Sph^1$-action.

On the other hand, $\Sph^3$ is naturally diffeomorphic to $SU(2)$, the group of unitary $2 \times 2$ matrices with determinant 1. 
The isometry between the Lie group $SU(2)$ and $\Sph^3$ is expressed as
\[
F:SU(2)\to \Sph^3\quad;\quad \begin{pmatrix}
       w_1  &       w_2  \\
-\bar{w_2}  &  \bar{w_1} \end{pmatrix}
\longmapsto (w_1,w_2).
\]
A basis for the Lie algebra of left invariant vector fields on $SU(2)$ is given by orthogonal matrices
\begin{equation*}
E:=E_1=\begin{pmatrix}
i  &  0  \\
0  &  -i
\end{pmatrix}, \quad
E_2=\begin{pmatrix}
0  &  1  \\
-1  &  0
\end{pmatrix},\quad
E_3=\begin{pmatrix}
0  &  i  \\
i  &  0
\end{pmatrix}
\end{equation*}
for $\mathfrak{su}(2)$. 
The differential $dF$ maps $B:=\{E_1,E_2,E_3\}$ to $\{\partial_{\xi_1} +\partial_{\xi_2} , \partial_\eta, \partial_{\xi_1} - \partial_{\xi_2}\}$.  
%
%
\subsection{The quantomorphism group}    
Consider $\Sph^3$ with the standard round metric \eqref{standard metric 1}. 
As we have seen, the Hopf fibration with this metric  is a Riemannian submersion. 
It is also a Boothby-Wang fibration, in the sense that $\pi^*\omega = d\theta_1$, where $\{\theta_i\}_{i=1,2,3}$ forms the dual basis to the vector fields $\{E_i\}_{i=1,2,3}$. 
For $\theta_1$ we need that $\theta_1\wedge d\theta_1$ coincides with the volume form $\mu$ generated by the metric $g^{\Sph^3}$.
Then the symplectic quotient of $\Sph^3$ via the circle action is $\Sph^2$. 
From here on we set $\theta:=\theta_1$. 

\begin{Def}
The  quantomorphism group of Sobolev class $s$ is defined by
\[
\cD_q^s(\Sph^3)=\{ \varphi\in\Sph^3\longrightarrow\Sph^3\mid \textrm{ $\varphi$  is an $H^s$ diffeomorphism and } \varphi^*\theta=\theta \}.
\]
\end{Def}
\cite[Cor.~2.7]{Ebin-Preston-arXiv} implies that if $s>\frac{5}{2}$ then $\cD_q^s(\Sph^3)$ is a $C^\infty$ submanifold of the space $\cD_{E,\mu}^s(\Sph^3)$ of all $H^s$ volume-preserving diffeomorphisms that commute with the flow of the Reeb vector field. 
Set 
\[
H_{E}^{s+1}(\Sph^3,\R)=\{ f\mid f:\Sph^3\longrightarrow\R    \textrm{  is $H^{s+1}$  and } Ef=0 \}.
\]
Then we have
\[
\cg:=T_{\id}\cD_q^s(\Sph^3)=\{ \Ss_{\theta} f \mid f:\Sph^3\longrightarrow\R ,\;  f\in H_{E}^{s+1}(\Sph^3,\R) \}
\]
where $\Ss_{\theta}f=u$ if and only if $\theta(u)=f$ and $i_ud\theta=-df$ (for more details, see \cite{Ebin-Preston-arXiv}). 
As derived in \cite[p.~20]{Ebin-Preston-arXiv}, the operator $S_{\theta}$ is given by
\begin{equation}\label{S theta}
\Ss_{\theta}f=fE-\frac{1}{2}(E_3f)E_2 +\frac{1}{2}(E_2f)E_3.    
\end{equation}
For any $x \in \Sph^3$, there exists an isomorphism between $\textrm{Ann}(E) \subset T^*_x\Sph^3$ and $\textrm{ker}(\theta) \subset T_x\Sph^3$, determined by $\Gamma(\theta_2) = -E_3$ and $\Gamma(\theta_3) = E_2$.

In the sequel, suppose that \[\rho:\Sph^2\to \R\] is a function in $H^{s}$. 
For convenience, we will use the same symbol $\rho$  to denote the lifted  function $\rho\circ\pi:\Sph^3\longrightarrow \R$. 
Let $\mathcal{A}:\mathfrak{X}(\Sph^3) \to \mathfrak{X}(\Sph^3)$ be the operator which acts on vector fields by
\begin{eqnarray*}
\mathcal{A}: u &\longmapsto & \rho^2 g^{\Sph^3}(u,E)E + g^{\Sph^3}(u,E_2)E_2 + g^{\Sph^3}(u,E_3)E_3 .
\end{eqnarray*}
Then $\mathcal{A}$ induces the positive definite inner product $\langle \; , \; \rangle^\A:\cg\times\cg\longrightarrow \R $ given by
\begin{eqnarray}\label{metric A}
\langle S_{\theta}{f}  ,  S_{\theta}{g} \rangle^{\A}      &=&  \langle f\rho^2 E-\frac{1}{2}(E_3f)E_2 +\frac{1}{2}(E_2f)E_3  ~,~  S_{\theta}{g} \rangle \nonumber\\
&=&  \langle f\rho^2 E-\frac{1}{2}\Gamma(df) ~,~  S_{\theta}{g} \rangle
\end{eqnarray}
%
%
\begin{Rem}
The principal $\Sph^1$-bundle  $\pi:\Sph^3\longrightarrow\Sph^2$ admits an $\mathfrak{s}^1 = T_e \Sph^1$-valued 1-form $\omega$ (see, for example, \cite[Cor.~6.11]{Lichtenfelz}). 
Using the connection form $\omega$ and following the formulation  \eqref{standard metric 2}, the Berger metric can be written  as
\begin{equation*}
g_{\alpha}^{\Sph^3}=\pi^*g^{\Sph^2} +\alpha^2\omega^*(dt)^2    
\end{equation*}
Practically, if we identify $\mathfrak{s}^1$ with $\R$, then for any $v, w \in T_p\Sph^3$, $\omega^*(dt)^2$ is given by $\omega_p(v)\omega_p(w)$. 
A class of non-standard metrics on the total  space $\Sph^3$ can be defined using a differentiable function 
\[
\rho:\Sph^2\to \R
\]
via 
\begin{equation}\label{non-standard metric 1}
g_{\rho}^{\Sph^3}=\pi^*g^{\Sph^2} +(\rho\circ\pi)\omega^*(dt)^2    
\end{equation}
For example, see \cite[p.~270]{Cabrerizo}. 
Equivalently, we can write
\begin{equation}\label{non-standard metric 2}
g^{\Sph^3}_{\rho}=\frac{1}{4} \big( d\lambda^2 +\sin^2\lambda  d\phi^2\Big) +   \rho(\lambda,\phi)\big(d\chi  - \frac{1}{2} \cos\lambda  d\zeta\big)^2 . 
\end{equation}
Depending on the range of $\rho$, the metric \eqref{non-standard metric 1} could be Lorenzian, semi-Riemannian or Riemannian. 
The Reeb field $E$ with respect to the metric~\eqref{non-standard metric 1} is a Killing vector field. 
This metric inspires the definition of the operator $\mathcal{A}$ in the weak inner product \eqref{metric A}.
\end{Rem}

In the next step, we compute the adjoint of the operator $S_{\theta}$ with respect to the $L^2$-type metric \eqref{metric A}.
%
%
\begin{prop}
The adjoint of the operator $\Ss_{\theta}$ with respect to the weak metric \eqref{metric A} exists and is given by
\begin{eqnarray*}
\Ss_{\theta,\A}^*: T_{\id}\cD_q^s(\Sph^3)  &\longrightarrow &  H_{E}^{s-1}(\Sph^3,\R)\\
aE+bE_2+cE_3  & \longmapsto &  \rho^2 a +\frac{1}{2}(E_3b) -\frac{1}{2}(E_2c)
\end{eqnarray*}
\end{prop}
\begin{proof}
For any $w=aE+bE_2+cE_3\in T_{\id}\cD_q^s(\Sph^3)$ and any $f\in  H_{E}^{s+1}(\Sph^3,\R)$ we have
\begin{eqnarray*}
\langle S_{\theta}{f}  ,  w \rangle^{\A}    &=& \int_{\Sph^3} g^{\Sph^3}(\A(\Ss_{\theta}{f} ) ,  w) d\mu\\
&=&  \int_{\Sph^3}  g^{\Sph^3}\big(  f\rho^2 E-\frac{1}{2}(E_3f)E_2 +\frac{1}{2}(E_2f)E_3  ~,~  aE+bE_2+cE_3   \big)d\mu\\
&=&  \int_{\Sph^3}  \Big(  f\rho^2 a -\frac{1}{2}(E_3f)b +\frac{1}{2}(E_2f)c   \Big)d\mu\\
&=&    \int_{\Sph^3}  f \Big(  \rho^2 a +\frac{1}{2}(E_3b) -\frac{1}{2}(E_2c)   \Big)d\mu
\end{eqnarray*}
which implies that 
\begin{equation}\label{S theta A star}
\Ss_{\theta,\A}^*(aE+bE_2+cE_3)=\rho^2 a +\frac{1}{2}(E_3b) -\frac{1}{2}(E_2c).
\end{equation}
In the next step, we show that  if $[E,w]=0$, then 
\[
E\big( \Ss_{\theta,\A}^*w\big) = 0
\]
which implies that $\Ss_{\theta,\A}^*w$ belongs to $H_{E}^{s-1}(\Sph^3,\R)$. 
Suppose $h:\Sph^3\longrightarrow\R$ is a differentiable function. 
Since $E$ is a Killing vector field for the metric $g^{\Sph^3}$ and $L_{E}\mathcal{A}=\mathcal{A}L_{E}$ we get
\begin{eqnarray*}
\int_{\Sph^3} h L_{E} \big( \Ss_{\theta,\A}^*w\big) d\mu   &=&  -\int_{\Sph^3} \big( \Ss_{\theta,\A}^*w\big) L_Eh   d\mu \\
&=&    -\int_{\Sph^3} g^{\Sph^3} \big( \Ss_{\theta}L_Eh , \mathcal{A}w\big)  d\mu\\
&=&     \int_{\Sph^3} - g^{\Sph^3} \big( L_E \Ss_{\theta}h , \mathcal{A}w\big)  d\mu \\
&=&     \int_{\Sph^3} - g^{\Sph^3} \big( L_E \Ss_{\theta}h , \mathcal{A}w\big) +Eg^{\Sph^3} \big(  \Ss_{\theta}h , \mathcal{A}w\big)  d\mu \\
&=&       \int_{\Sph^3}  g^{\Sph^3} \big(  \Ss_{\theta}h , L_E\mathcal{A}w\big)d\mu\\
&=&       \int_{\Sph^3}  g^{\Sph^3} \big(  \Ss_{\theta}h , \mathcal{A}L_Ew\big)d\mu=0.
\end{eqnarray*}
Since the function $h$ is arbitrary, we conclude that $L_{E} \big( \Ss_{\theta,\A}^*w\big) = 0$ for any $E$-invariant vector field $w$.
\end{proof}
Now, we define the contact Laplacian by
\[
\Delta_{\theta,\A}: H_{E}^{s+1}(\Sph^3,\R)  \longrightarrow   H_{E}^{s-1}(\Sph^3,\R)\quad ; \quad f \longmapsto   S_{\theta,\A}^* S_{\theta} f
\]
\begin{prop}
The contact Laplacian is given by
\begin{eqnarray*}
\Delta_{\theta,\A}: H_{E}^{s+1}(\Sph^3,\R)  &\longrightarrow &  H_{E}^{s-1}(\Sph^3,\R)\\
f  & \longmapsto & (\rho^2   - \Delta) f 
\end{eqnarray*}
where $\Delta$ denote the Laplacian on the sphere $\Sph^2$ lifted to $\Sph^3$.
Moreover $\Delta_{\theta,\A}$ reduces to an invertible elliptic operator on $\Sph^2$.
\end{prop}
\begin{proof}
For any $h\in H_{E}^{s+1}(\Sph^3,\R)$, we observe that
\begin{eqnarray*}
\Delta_{\theta,\A} f   &=&  S_{\theta,\A}^* S_{\theta} f     \\
&=&   S_{\theta,\A}^*  \Big(  f  E-\frac{1}{2}(E_3f)E_2 +\frac{1}{2}(E_2f)E_3\Big)  \\
&=&  f \rho^2   - \frac{1}{4}  E_3^2f - \frac{1}{4}E_2^2f = f \rho^2   - \Delta f
\end{eqnarray*}    
where $\Delta$ denotes the Laplacian on the two-dimensional sphere with radius $\frac{1}{2}$. Moreover,
\begin{eqnarray*}
\langle \A \Ss_{\theta} f , \Ss_{\theta} f\rangle   &=&  \int_{\Sph^3}     f\Delta_{\theta,\A} f d\mu        \\
&=&   \int_{\Sph^2}      ( f^2 \rho^2  - f\Delta f ) d\nu   \\
&=&   \int_{\Sph^2}      ( f^2 \rho^2  + |\nabla f|^2 ) d\nu\geq 0
\end{eqnarray*}    
and 
\[
\int_{\Sph^2}      ( f^2 \rho^2  + |\nabla f|^2 ) d\nu= 0 
\]
if and only if $f=0$. As a result, if we reduce the contact Laplacian to $\Sph^2$, we obtain
\[
\Delta_{\theta,\A}:H^{s+1}(\Sph^2,\R) \longrightarrow H^{s-1}(\Sph^2,\R)\quad;\quad f\longmapsto \rho^2f- \Delta f
\]
which is clearly elliptic.
\end{proof}
\begin{Rem}
The invertibility of the operator $\Delta_{\theta,\A} = \rho^2 - \Delta$ can also be proved using the strong maximum principle, see for example \cite[Cor.~2.11]{Kazdan}. 
Furthermore, this type of Schr\"{o}dinger operator has been studied by Gurarie~\cite{Guraie}.
\end{Rem}
%
%
%
\begin{Lem}
Let $f:\Sph^3\longrightarrow\R$  be $E$-invariant. 
Moreover, suppose $d\nu$ is the volume form for the standard metric on $\Sph^2$ with radius $\frac{1}{2}$. 
Then, for $\tilde{f}:\Sph^2\longrightarrow\R$ defined by $\tilde{f}(\pi(x))=f(x)$, we have $\int_{\Sph^3}fd\mu = 2\pi\int_{\Sph^2} \tilde{f}  d\nu$.  
\end{Lem}
\begin{proof}
Consider the coordinates $\lambda=2\eta$, $\chi=\frac{\xi_1+\xi_2}{2}$ and $\zeta=\xi_1-\xi_2$. Then $0\leq\lambda\leq \pi$, $0\leq\chi\leq2\pi$ and $-2\pi\leq\zeta\leq2\pi$. Moreover the volume form on $\Sph^3$ with respect to the standard metric \eqref{standard metric 1} is 
\[d\mu=-\sin\eta\cos\eta d\xi_1\wedge d\xi_2 \wedge d\eta\]
and the $E$-invariant functions are functions of $\lambda$ and $\zeta$. As a result we get
\begin{eqnarray*}
\int_{\Sph^3} fd\mu &=& -\int_0^{\frac{\pi}{2}} \int_0^{2\pi}\int_0^{2\pi}f(\xi_1,\xi_2,\eta)\sin\eta\cos\eta d\xi_1\wedge d\xi_2 \wedge d\eta\\
&=& 2\pi \int_0^{\pi} \int_{-2\pi}^{2\pi} \tilde{f}(\lambda,\zeta)\frac{1}{4}\sin\lambda d\zeta\wedge d\lambda  
= 2\pi \int_{\Sph^2}\tilde{f} d\nu .
\end{eqnarray*}
\end{proof}
The fibration $\pi:\Sph^3\longrightarrow\Sph^2$ induces the projection 
\begin{eqnarray}\label{Projection}
\Pi:\cD^s_q(\Sph^3)&\longrightarrow & \cD^s_{\nu}(\Sph^2)\\
\varphi &\longmapsto& \Big(\tilde{\varphi}(\pi(x))=\pi\circ\varphi(x)\Big)\nonumber
\end{eqnarray}
where 
\[
\cD^s_{\nu}(\Sph^2)=
\{ \varphi\in\Sph^2\longrightarrow\Sph^2; \textrm{ $\varphi$  is a $H^s$ symplectic diffeomorphism }  \}.
\]
Since every quantomorphism is $E$-invariant, the map $\Pi$ is well-defined. 
Following Ebin and Preston~\cite[Thm.~4.4]{Ebin-Preston-arXiv}, we propose next a metric on the base space $\cD^s_{\nu}(\Sph^2)$ that makes $\Pi$ a Riemannian submersion.

We first note that  $\Delta_{\theta ,\mathcal{A}}1=\rho^2$. 
For the Hamiltonian vector fields $u=\nabla^\perp\tilde{f}$ and $v=\nabla^\perp\tilde{g}$ in $\tilde{\cg}:=T_e\cD^s_{\nu}(\Sph^2)$ we then define the (weak) inner product
\begin{eqnarray}\label{metric-Ham}
    \llangle u,v\rrangle = 2\pi\int_{\Sph^2}\tilde{f} \big((\rho^2-\Delta)\tilde{g}\big) d\nu  
    -  2\pi\frac{    \int_{\Sph^2}\tilde{f} \rho^2 d\nu   \int_{\Sph^2}\tilde{g} \rho^2 d\nu    }{    \int_{\Sph^2} \rho^2 d\nu    } .
\end{eqnarray}
This inner product is extended to a right-invariant metric on  $\cD^s_{\nu}(\Sph^2)$. 
Here, ``weak'' means that the topology induced by this metric is weaker than the $H^s$ topology. 

The following result is a modification of \cite[Thm~4.4 and Prop.~4.5]{Ebin-Preston-arXiv}.
%
%
%
%
\begin{The}\phantom{hej}
\begin{enumerate}[i.]
    \item     The map $\Pi : (\cD^s_{\textrm{q}}(\Sph^3), \langle \;,\;\rangle^\A)\longrightarrow (\cD^s_{\nu}(\Sph^2),\llangle\;,\;\rrangle)$ is a Riemannian submersion.
    \item The Euler-Arnold equation of $(\cD^s_{\nu}(\Sph^2),\llangle\;,\;\rrangle)$ is given by
    \begin{eqnarray}\label{Geo-eq-S2}
     \partial_t(\rho^2-\Delta)\tilde{f} + \{\tilde{f} ,  (\rho^2-\Delta)\tilde{f}\}=0.
    \end{eqnarray}
\end{enumerate}
\end{The}

\begin{proof}
\phantom{hej}

\textit{i}.    For $\tilde{f}, \tilde{g} \in H^{s}(\Sph^2,\R)$, consider $f, g \in H^{s}(\Sph^3,\R)$ defined by $f(x) = \tilde{f}(\pi(x))$ and $g(x) = \tilde{g}(\pi(x))$. 
Set $\bar{f} = f + a$ and $\bar{g} = g + b$, where $a$ and $b$ are real constants chosen to ensure that $\Ss_{\theta}\bar{f}$ and $\Ss_{\theta}\bar{g}$ are horizontal. Specifically, we select $a$ and $b$ such that
\[
\langle \Ss_{\theta}\bar{f},E\rangle_{\mathcal{A}}  =  \langle \Ss_{\theta}\bar{g},E\rangle_{\mathcal{A}}  = 0.
\]
Note that, $\langle \Ss_{\theta}\bar{f} , E\rangle_{\mathcal{A}}=0$ if and only if 
\begin{eqnarray*}
\int_{\Sph^3}g^{\Sph^3} \big( \mathcal{A}\Ss_\theta(f+a),\Ss_\theta 1 \big) d\mu  &=&  \int_{\Sph^3}(f+a)\Delta_{\theta,\mathcal{A}}1 d\mu\\
&=& \int_{\Sph^3}f\rho^2 d\mu    + a \int_{\Sph^3}\rho^2 d\mu\\
&=&  2\pi\big( \int_{\Sph^2}\tilde{f}{\rho}^2 d\nu    +   a \int_{\Sph^2}{\rho}^2 d\nu \big) = 0.
\end{eqnarray*}
This implies that $a=-\frac{ \int_{\Sph^2}\tilde{f}\rho^2 d\nu     }{   \int_{\Sph^2}\rho^2 d\nu}$. Similarly, for $\bar{g}$, we have
$b=-\frac{ \int_{\Sph^2}\tilde{g}\rho^2 d\nu     }{   \int_{\Sph^2}\rho^2 d\nu}$. As a result we have
\begin{eqnarray*}
&& \langle \Ss_{\theta}\bar{f}  ,  \Ss_{\theta}\bar{g} \rangle_{\mathcal{A}}    \\
&=&\int_{\Sph^3} f\Delta_{\theta,\mathcal{A}} g d\mu + a\int_{\Sph^3} g\Delta_{\theta,\mathcal{A} }1 d\mu   + b\int_{\Sph^3}f\Delta_{\theta,\mathcal{A}}1 d\mu   + ab \int_{\Sph^3}\Delta_{\theta,\mathcal{A}}1 d\mu\\
&=&   2\pi\int_{\Sph^2} \tilde{f}(\rho^2-\Delta) \tilde{g}  d\nu  
-2\pi\frac{ \int_{\Sph^2}\tilde{f} \rho^2 d\nu     }{   \int_{\Sph^2}\rho^2 d\nu}   \int_{\Sph^2} \tilde{g} \rho^2 d\nu   
 -2\pi\frac{ \int_{\Sph^2}\tilde{g}\rho^2 d\nu     }{   \int_{\Sph^2}\rho^2 d\nu}   \int_{\Sph^2}\tilde{f} \rho^2 d\nu\\
&& + 2\pi\frac{     \int_{\Sph^2}\tilde{f} \rho^2 d\nu    \int_{\Sph^2}\tilde{g}\rho^2 d\nu     }{   \int_{\Sph^2}\rho^2 d\nu}   \\
&=&   2\pi\int_{\Sph^2} \tilde{f}(\rho^2-\Delta) \tilde{g}  d\nu  
-2\pi\frac{ \int_{\Sph^2}\tilde{f}\rho^2 d\nu     }{   \int_{\Sph^2}\rho^2 d\nu}   \int_{\Sph^2} \tilde{g} \rho^2 d\nu   =   \llangle u,v\rrangle
\end{eqnarray*}

\textit{ii.} To prove this part, we will write the Euler-Arnold equation  $\partial_t u+\ad^*_uu=0$ for  $\cD^s_{\textrm{Ham}}(\Sph^2)$ in the context of the metric given by  \eqref{metric-Ham}. Consider $u=\nabla^\perp \tilde{f}, v=\nabla^\perp\tilde{g}, w=\nabla^\perp\tilde{h}\in\tilde{\cg}$. Then we have
\begin{eqnarray*}
\llangle \ad^*_uv,w\rrangle   &=&   \llangle v,\ad_uw\rrangle = -\llangle v,\nabla^\perp\{\tilde{f},\tilde{h}\}\rrangle  \\
&=& - 2\pi \int_{\Sph^2}  \tilde{g} \big(  ( \rho^2-\Delta)\{\tilde{f},\tilde{h}\} \big)   d\nu  
+ 2\pi\frac{  \int_{\Sph^2}\tilde{g} \rho^2 d\nu   \int_{\Sph^2}\{\tilde{f},\tilde{h}\} \rho^2 d\nu }{ \int_{\Sph^2} \rho^2 d\nu}\\
&=&  2\pi \int_{\Sph^2}   \{\tilde{f} , (\rho^2-\Delta)\tilde{g}\}\tilde{h}   d\nu  
-  2\pi\frac{    \int_{\Sph^2}\tilde{g} \rho^2  d\nu   \int_{\Sph^2}\{\tilde{f}, \rho^2\} \tilde{h}   d\nu    }{    \int_{\Sph^2} \rho^2 d\nu    }\\
&=&  2\pi\int_{\Sph^2} \Big(  \{\tilde{f},(\rho^2-\Delta)\tilde{g}\}  -  \frac{    \int_{\Sph^2}\tilde{g} \rho^2 d\nu       }{    \int_{\Sph^2} \rho^2 d\nu    }       \{\tilde{f}, \rho^2\}   \Big)\tilde{h}   d\nu.
\end{eqnarray*}
As a result for the time dependent vector field $u$ and arbitrary $w\in\tilde{\cg}$ we get
\begin{eqnarray*}
\llangle \partial_tu+\ad^*_uu,w\rrangle   &=&  2\pi \int_{\Sph^2} \Big(   \partial_t (\rho^2-\Delta )\tilde{f}   
-  \frac{    \int_{\Sph^2}\partial_t\tilde{f} \rho^2 d\nu       }{    \int_{\Sph^2} \rho^2 d\nu    }   \rho^2         \\
&&  +   \{\tilde{f} , (\rho^2-\Delta)\tilde{f}\}    -  \frac{    \int_{\Sph^2}\tilde{f} \rho^2 d\nu       }{    \int_{\Sph^2} \rho^2 d\nu    }       \{\tilde{f}, \rho^2\}   \Big)    \tilde{h}   d\nu.
\end{eqnarray*}
which means that $\partial_tu+\ad^*_uu=0$ if and only if
\begin{eqnarray*}
\partial_t(\rho^2-\Delta)\tilde{f}     +  \{\tilde{f},(\rho^2-\Delta)\tilde{f}\}
-  \frac{   \partial_t \int_{\Sph^2}\tilde{f} \rho^2 d\nu       }{    \int_{\Sph^2} \rho^2 d\nu    } \rho^2          
-  \frac{    \int_{\Sph^2}\tilde{f} \rho^2 d\nu       }{    \int_{\Sph^2} \rho^2 d\nu    } 
\{\tilde{f}, \rho^2 \} = 0
\end{eqnarray*}
It is always possible to choose $c(t):\R\longrightarrow\R$ and replace the Hamiltonian function $ \tilde{f}=\tilde{f}(t)$ with $\tilde{f}(t)+c(t)$  such that $\int_{\Sph^2}(\tilde{f} +c )\rho^2 d\nu=0$. Notice that $u=\nabla^\perp\tilde{f}=\nabla^\perp(\tilde{f}+c)$
which means that we can always choose a  Hamiltonian function for $u$ with the property $\int_{\Sph^2}\tilde{f} \tilde{\phi} \bar{\rho}  d\nu=0 $. As a result, the Euler-Arnold equation becomes
\begin{equation*}
\partial_t (\rho^2 - \Delta) \tilde{f} +  \{\tilde{f},(\rho^2 -\Delta)\tilde{f}\}=0.
\end{equation*}
\end{proof}

\section{Local and global well-posedness}\label{sec:analysis}
%
%

\begin{The}(\textbf{Local well-posedness})
Geodesics of $(\cD^s_q(\Sph^3),\langle\;,\;\rangle_{\mathcal{A}})$ exists locally in the sense of the Picard-Lindelöf theorem.
That is, for any choice of initial conditions, there exists a (non-empty) maximal time interval $(-T_a,T_b)$ for which a solution exists, is unique, and depends smoothly on the initial data.
\end{The}
\begin{proof}
Following \cite{Ebin-Marsden} and \cite{Ebin-Preston} we show that the geodesics correspond to the integral curves of a smooth vector field on $\cD^s_q(\Sph^3)$. For $f\in H^{s+1}_E(\Sph^3,\R) $ define $m(t):=\Delta_{\theta,\mathcal{A}}\Ss_\theta f(t)$. Consider the vector field
\begin{eqnarray*}
Z:\cD^s_q(\Sph^3)\longrightarrow T\cD^s_q(\Sph^3)\quad;\quad \eta\longmapsto (\Ss_\theta(\Delta_{\theta,\mathcal{A}})^{-1})_\eta m_0
\end{eqnarray*}
where 
\[
(\Ss_\theta(\Delta_{\theta,\mathcal{A}})^{-1})_\eta=dR_\eta (\Ss_\theta(\Delta_{\theta,\mathcal{A}})^{-1})\circ (R_\eta)^{-1}m(0)
\]
and $u(t)=\Ss_\theta f(t)=\dot{\eta}(t)\circ\eta^{-1}(t)$. Before demonstrating that this vector field is smooth, we first show that the integral curves of $Z$ correspond to the solutions of 
\begin{eqnarray*}
\partial_t (\Delta_{\theta,\mathcal{A}} f(t) \circ\eta(t) ) &=& 
\Big( \partial_t \Delta_{\theta,\mathcal{A}} f(t) + \Ss_\theta f(t) \big(\Delta_{\theta,\mathcal{A}} f(t)\big)       \Big)\circ\eta(t) \\
&=&  \Big( \partial_t \Delta_{\theta,\mathcal{A}} f(t) + \{ f(t) ,\Delta_{\theta,\mathcal{A}} f(t)\}       \Big)\circ\eta(t) = 0
\end{eqnarray*}
Let $\eps>0$ and $\eta:(-\eps,\eps)\longrightarrow \cD^s_q(\Sph^3) $ be an integral curve of $Z$. The $\dot{\eta}(t)=\Ss_\theta f(t)\circ\eta(t)=Z(\eta(t))$. 
Then
\begin{eqnarray*}
&&                        \Ss_\theta f(t)\circ\eta(t) =  dR_\eta (\Ss^*_{\theta,\mathcal{A}})^{-1} (R_\eta)^{-1}m(0)\\
&\Longleftrightarrow&     dR_\eta\Ss_\theta f(t)      =  dR_\eta (\Ss^*_{\theta,\mathcal{A}})^{-1} (R_\eta)^{-1}m(0)\\
&\Longleftrightarrow&     \Ss_\theta f(t)    =  (\Ss^*_{\theta,\mathcal{A}})^{-1} (R_\eta)^{-1}m(0)\\
&\Longleftrightarrow&     \Ss^*_{\theta,\mathcal{A}} \Ss_\theta f(t)    =   m(0)\circ\eta(t)^{-1}\\
&\Longleftrightarrow&     \Ss^*_{\theta,\mathcal{A}} \Ss_\theta f(t)    =   m(0)\circ\eta(t)^{-1}\\
&\Longleftrightarrow&     \Delta_{\theta,\mathcal{A}} f(t)              =   m(0)\circ\eta(t)^{-1}\\
&\Longleftrightarrow&     \Delta_{\theta,\mathcal{A}} f(t) \circ\eta(t) =   m(0)\\
&\Longleftrightarrow&     \partial_t (\Delta_{\theta,\mathcal{A}} f(t) \circ\eta(t) )=  0\\
\end{eqnarray*}
which is true for any geodesic of the metric $(\cD^s_q(\Sph^3),\langle\;,\;\rangle_{\mathcal{A}})$.

On the other hand, for a first-order differential operator $X$, the operator
\begin{equation}\label{smooth operator}
X_\eta(h)(p):=X(h\circ\eta^{-1})\circ\eta|_p
\end{equation}
depends smoothly on $\eta\in\cD^s(\Sph^3)$ and $h\in H^s(\Sph^3)$ (see, e.g., \cite{Ebin-Marsden} or \cite[Thm.~3.1]{Ebin-Preston}).

In the next step, we show that  
\[
(\Ss^*_{\theta,\mathcal{A}})_\eta : T_\eta \cD^s(\Sph^3) \longrightarrow H^{s-1}(\Sph^3,\R) \quad ; \quad  
v_\eta \longmapsto R_\eta (\Ss^*_{\theta,\mathcal{A}}) \circ (dR_\eta)^{-1}v_\eta  
\]  
is smooth. Note that generally for $w=aE+bE_2+cE_3\in T_e\cD^s(\Sph^3)$ we have
\[
(\Ss^*_{\theta,\mathcal{A}})(w)= \rho^2 a +\frac{1}{2}  E_3 b - \frac{1}{2}E_2c 
\]  
More precisely, for  any differentiable function $f:\Sph^3\longrightarrow \R$ we have
\begin{eqnarray*}
&& \int_{\Sph^3}f(\Ss^*_{\theta,\mathcal{A}})(w)  d\mu  =   \int_{\Sph^3}g^{\Sph^3}\big(  \mathcal{A}\Ss_{\theta}f , w  \big)  d\mu\\
&=& \int_{\Sph^3}g^{\Sph^3}\big(  fh\rho^2 E - \frac{1}{2}  (E_3f)E_2  + \frac{1}{2}  (E_2f)E_3 , aE+bE_2+cE_3  \big)  d\mu\\
&=&   \int_{\Sph^3} \Big( fa\rho^2 g^{\Sph^3}\big( E ,E \big) - \frac{1}{2}  (E_3f)b  g^{\Sph^3}\big(   E_2 ,E_2  \big) + \frac{1}{2}(E_2f) cg^{\Sph^3}\big( E_3 ,E_3 \big)  \Big)  d\mu  \\
&=&   \int_{\Sph^3}  fa\rho^2d\mu  - \frac{1}{2} \int_{\Sph^3}  (E_3f)bd\mu  + \frac{1}{2}\int_{\Sph^3} (E_2f) c   d\mu  \\
&=&   \int_{\Sph^3}  fa\rho^2d\mu  + \frac{1}{2} \int_{\Sph^3} \Big( f(E_3b)+fb\mathrm{div}(E_3)  \Big)d\mu   - \frac{1}{2}\int_{\Sph^3} \Big( f(E_2c)+fc\mathrm{div}(E_2)  \Big)   d\mu  \\
&=&   \int_{\Sph^3}  \Big(  fa\rho^2  + \frac{1}{2}f(E_3b) - \frac{1}{2}f(E_2c)  \Big) d\mu  \\
&=&   \int_{\Sph^3}  f\Big(  a\rho^2  + \frac{1}{2}(E_3b) - \frac{1}{2}(E_2c)  \Big) d\mu.
\end{eqnarray*}
As a result for the operator $(\Ss^*_{\theta,\mathcal{A}})_\eta$ we have
\begin{eqnarray*}
&&   R_\eta (\Ss^*_{\theta,\mathcal{A}}) \circ (dR_\eta)^{-1}\Big( aE\circ\eta + bE_2\circ\eta + cE_3\circ\eta \Big)\\
&=&  R_\eta (\Ss^*_{\theta,\mathcal{A}}) \Big( a\circ\eta^{-1}E + b\circ\eta^{-1}E_2 + c\circ\eta^{-1}E_3 \Big)\\
&=&  R_\eta  \Big( \rho^2a\circ\eta^{-1} + \frac{1}{2}E_3(b\circ\eta^{-1})  -  \frac{1}{2} E_2(c\circ\eta^{-1}) \Big)\\
&=&   (\rho^2\circ\eta) a + \frac{1}{2}E_3(b\circ\eta^{-1})\circ\eta  -  \frac{1}{2} E_2(c\circ\eta^{-1})\circ\eta 
\end{eqnarray*}
which is smooth by \eqref{smooth operator}. Since $(\Ss^*_{\theta,\mathcal{A}})_\eta$ is smooth on $T_\eta \cD^s(\Sph^3)$, it remains smooth when restricted to \(T_\eta \cD^s_q(\Sph^3)\), as \(\cD^s_q(\Sph^3) \subseteq \cD^s(\Sph^3)\) is a smooth submanifold. On this submanifold, the operator \((\Ss^*_{\theta,\mathcal{A}})_\eta\) becomes an isomorphism from \(T_\eta \cD^s_q(\Sph^3)\) to \(H^{s-1}_E(\Sph^3,\R)\).  Furthermore, the operator  
\[
\Delta_{\theta,\mathcal{A}} : H^{s+1}_E(\Sph^3,\R) \longrightarrow H^{s-1}_E(\Sph^3,\R)  
\]  
is linear and smooth. The same holds true for the inverse of \(\Delta_{\theta,\mathcal{A}}\).  As a result,  
\[
(\Ss^*_{\theta,\mathcal{A}})^{-1}_\eta(h) = (\Ss_\theta (\Delta_{\theta,\mathcal{A}})^{-1}_\eta)(h)  
\]  
is also smooth. This implies that the vector field \(Z\) is smooth.  

Since \(Z\) is a smooth vector field on a Banach manifold, its integral curves exist by the standard theory of ordinary differential equations on Banach manifolds. Moreover, the solutions depend smoothly on the initial values.  
\end{proof}

%
%
\begin{The}\label{Global existence}(\textbf{Global well-posedness})
For initial data $f_0\in H^{s+1}(\Sph^2,\R)\simeq H^{s+1}_{E}(\Sph^3,\R)$ with $s>2$,  the solution of equation \eqref{Geo-eq-S2} exists for all time.
\end{The}
\begin{proof}
The idea for the proof originates from \cite[Thm.~4.6]{Ebin-Preston-arXiv} and \cite[Thm.~3.3]{Ebin-Preston}, with minor modifications. 
In particular, we shall refer to the computations in \cite[Thm.~4.6]{Ebin-Preston-arXiv} at various steps. 

Let $(-T_a, T_b)$ denote the maximum time interval for the $H^{s+1}$ time-dependent stream function
\[
f(t):\Sph^2 \longrightarrow \R
\]
which solves equation~\eqref{Geo-eq-S2}.
If $T_a$ is finite, then $\lim_{t\rightarrow T_a} \|f(t)\|_{H^s+1}=\infty$. 
The same holds for $-T_b$. 
To show the global existence, we prove that the norm $\|f(t)\|_{H^s+1}$, or equivalently $\|u(t)\|_{s}$ with $u(t)=\nabla^\perp f(t)$, remains bounded. 
To this end, we consider three objects: the $H^{s+1}$ stream function $f$, the $H^s$ velocity field $u(t)=\nabla^\perp f(t)$, and the corresponding Hamiltonian  diffeomorphisms $\eta_t$. 
Define the quantity
\[
m(t)=(\rho^2-\Delta )f(t). 
\]
Suppose $f(t)$ is a solution of \eqref{Geo-eq-S2}. 
Then 
\begin{eqnarray*}
\partial_t \big( m(t)\circ \eta(t)\big)  &=&   (\rho^2-\Delta )\partial_tf(t)\circ\eta(t) + dm(t)\partial_t\eta(t)\\
&=& \Big(  (\rho^2-\Delta )\partial_tf(t) + u(t)m(t) \Big) \eta(t)\\
&=& \Big(  (\rho^2-\Delta )\partial_tf(t) + \{f,(\rho^2 -\Delta) f  \}\Big) \eta(t) = 0, 
\end{eqnarray*}
which implies that the quantity $ m(t)\circ \eta(t)=m(0)$ is conserved. This conserved quantity will be used to show that the corresponding velocity field $u(t)$  is quasi-Lipschitz. More precisely for the inverse of the elliptic operator 
\[
\Delta_{\theta,\mathcal{A}}:H^{s+1}(\Sph^2,\R) \longrightarrow H^{s-1}(\Sph^2,\R)\quad ;\quad f\mapsto (\rho^2 -\Delta) f  
\]
consider the Schwartz kernel\footnote{We recall that, the Schwartz kernel of an elliptic invertible operator \( P \) on a compact manifold \( N \)  is a distributional kernel $k_P:N\times N \longrightarrow \R$, which acts by integrating against functions \( f \) on \( M \) via:
\[
P f(x) = \int_N k_P(x, y) f(y) \, d\text{vol}_N(y),
\]
where $d\text{vol}_N $ is the Riemannian volume form on $ N $. $ k_P(x, y) $ is smooth  on $N\times N\setminus \textrm{Diag}$, where $\textrm{Diag} = \{(x, x) \mid x \in N \}$ is the diagonal.  
}        
$k:\Sph^2\times\Sph^2 \longrightarrow \R$ defined by
\[
f(x)=\int_{\Sph^2}k(x,y) \Delta_{\theta,\mathcal{A}} f(y)dy.
\]
For more details see \cite[Ch.~7]{taylor} and \cite[p.~28]{Ebin-Preston-arXiv}.

In this step, we show that $f$ is bounded uniformly in time. 
Since $f$ is a solution of \eqref{Geo-eq-S2}, then $\int_{\Sph^2} \Delta_{\theta,\mathcal{A}} f$ is constant in time. Furthermore,
\[
\int_{\Sph^2} \Delta_{\theta,\mathcal{A}} f d\nu = \int_{\Sph^2} f\Delta_{\theta,\mathcal{A}} 1 d\nu =\int_{\Sph^2} f\rho^2 d\nu\geq \min f\int_{\Sph^2} \rho^2 d\nu .
\]  
Suppose that $|df|<K$. Then the bound
\begin{equation}
\max{f}-\min{f}<\pi K
\end{equation}
implies that $\max f$ is also bounded and consequently  $f$ is bounded uniformly in time.

Consider, around a point of $\Sph^2$, a normal neighborhood with radius $\pi$ (the injectivity radius for the sphere with radius $1$). 
There is a method to reconstruct the differential $df$ from $\Delta_{\theta,\mathcal{A}} f$, as utilized in various sources, including Theorem 2.5.1 of \cite{Morrey}, Lemma 1.4 of \cite{Kato}, Lemma 6.3 of \cite{Ebin1}, Theorem 1.1 of \cite{Benn-Suri}, \cite{Ebin-Preston-arXiv} and Theorem 3.3 of \cite{Ebin-Preston}. 
In fact, replacing $\Delta_\theta$ in \cite{Ebin-Preston-arXiv} with the operator $\Delta_{\theta,\mathcal{A}}$ we observe that 
\begin{equation}\label{quasi-Lipschitz bound}
|df^\prime(y)-df(y)|\leq K |x-y|\log \big(\frac{\pi}{|x-y|}\big)
\end{equation}
where $|x-y|$ denotes the distance on sphere and  $\pi$ represents  the diameter of the manifold $\Sph^2$. We recall that for $y\in\Sph^2$ with $|x-y|< \pi$,  $df^\prime(y)\in T_y\Sph^2$ is the parallel transport along a parametrized geodesic $\chi$ with the property $\chi(0)=x$ and $\chi(1)=y$, see \cite[pp.29-30]{Ebin-Preston-arXiv}. 
If we divide both sides of \eqref{quasi-Lipschitz bound} by $|x-y|^\gamma$ for  $0<\gamma<1$ then
\begin{equation*}
\frac{|df^\prime(y)-df(y)|}{|x-y|^\gamma}   \leq K |x-y|^{\alpha}\log \big(\frac{\pi}{|x-y|}\big)
\end{equation*}
where $\alpha=1-\gamma$. But $\log \big(\frac{\pi}{|x-y|}\big)$ grows slower than $|x-y|^{\alpha}$. As a result $df$ (or equivalently $u = \nabla^\perp f$) is $C^\alpha$ \footnote{     Locally, for $k\in\mathbb{N}\cup\{0\}$, the   $C^{k+\alpha}$ norm of a diffeomorphism $\eta$ on a manifold $M$  is defined as follows
\[
|\eta|_{C^{k+\alpha}} = \max_{j \leq k} |D^j\eta|_\infty + 
 \sup_{x \neq y} \frac{|D^k\eta(x) - D^k\eta(y)|}{|x- y|^\alpha},
\]
where $\|D^j\eta\|_\infty$ is the sup-norm of the  $D^j\eta$.}. The parameter $K$ in the previous estimate is independent of time and by increasing $K$ and compactness of $\Sph^2$, we can also drop the condition $|x-y|< \pi$. 
According to \cite[Sec.~2.5]{Kato}, there then exist positive constants $C$ and $\beta$ such that the flow satisfies
\[
|\eta(t)(x)-\eta(t)(y)|\leq C|x-y|^\beta
\]
indicating that the flow $\eta$ is $C^\beta$.  
In the next step, we demonstrate that $\eta^{-1}$ is also H\"{o}lder continuous. 
For $t_0 \in [0, T_e)$, define $v(t) := -u(t_0 - t)$, with $\sigma$ as the flow associated with $v$. It follows that $\sigma(t_0) = \eta^{-1}(t_0)$. 
As shown in \cite[pp.~30-31]{Ebin-Preston-arXiv}, the flow $\sigma$ is $C^\alpha$, confirming the H\"{o}lder continuity of $\eta^{-1}$.  
From the equation
\[
\Delta_\rho f(t)=\Delta_\rho f(0)\circ \sigma(t)
\]
and the fact that $f(0)\in H^{s+1}(\Sph^2,\R)$, we conclude that $$\Delta_\rho f(0)\in H^{s-1}(\Sph^2,\R)\subseteq C^{\alpha}(\Sph^2,\R).$$ 
Moreover, we proved that  $\sigma$ is $C^\alpha$ as well. As a result, $\Delta_\rho f(t)$ is also $C^\alpha(\Sph^2,\R)$ independent of $t$. Since $\Delta_\rho$ is elliptic, the stream function $f(t)$ is bounded in $C^{2+\alpha}(\Sph^2,\R) $ and $u=\nabla^\perp f$ is bounded in $C^{1+\alpha}(T\Sph^2)$.

Finally, we show $f(t)$ to be bounded in the $H^{s+1}$ topology, or equivalently that $\Delta_{\theta,\mathcal{A}} f(t)$ is bounded in $H^{s-1}.$ We note that
$$ \partial_t \Delta_{\theta,\mathcal{A}} f = -\nabla^\perp f (\Delta_{\theta,\mathcal{A}} f)$$
so
\begin{eqnarray*}
 \partial_t \int_{\Sph^2}  (\Delta_{\theta,\mathcal{A}} f)^2 d\nu  &=&  \int_{\Sph^2} \partial_t  (\Delta_{\theta,\mathcal{A}} f)^2 d\nu\\
&=&  2 \int_{\Sph^2}  (\partial_t  \Delta_{\theta,\mathcal{A}} f)(\Delta_{\theta,\mathcal{A}} f) d\nu\\
&=&  - 2 \int_{\Sph^2}  (\nabla^\perp f  \Delta_{\theta,\mathcal{A}} f)(\Delta_{\theta,\mathcal{A}} f) d\nu\\
&=&  -  \int_{\Sph^2}  \nabla^\perp f  (\Delta_{\theta,\mathcal{A}} f)^2 d\nu\\
&=&  -  \int_{\Sph^2} \textrm{div} \big( (\Delta_{\theta,\mathcal{A}} f)^2 \nabla^\perp f\big)   d\nu   +  \\ && \int_{\Sph^2} (\Delta_{\theta,\mathcal{A}} f)^2 \textrm{div}( \nabla^\perp f)   d\nu  = 0.
\end{eqnarray*}
Moreover, taking $s-1$ spatial derivatives
\begin{eqnarray*}
\partial_t \int_{\Sph^2} |\nabla^{s-1} \Delta_{\theta,\mathcal{A}} f |^2 d\nu & = &  2\int_{\Sph^2} \langle \nabla^{s-1} \partial_t \Delta_{\theta,\mathcal{A}} f , \Delta_{\theta,\mathcal{A}} f  \rangle   d\nu\\
&=& -2 \int_{\Sph^2} \langle  \nabla^{s-1}  \nabla_{\nabla^\perp f} \Delta_{\theta,\mathcal{A}} f, \nabla^{s-1} \Delta_{\theta,\mathcal{A}} f\rangle d\nu\\
&=&-2 \int_{\Sph^2} \langle \nabla_{\nabla^\perp f} \nabla^{s-1} \Delta_{\theta,\mathcal{A}} f, \nabla^{s-1} \Delta_{\theta,\mathcal{A}} f\rangle d\nu\\
&& -2 \int_{\Sph^2} \langle [\nabla^{s-1}  ,  \nabla_{\nabla^\perp f} ] \Delta_{\theta,\mathcal{A}} f, \nabla^{s-1} \Delta_{\theta,\mathcal{A}} f\rangle d\nu\\
&=& - \int_{\Sph^2} \nabla_{\nabla^\perp f} \langle \nabla^{s-1} \Delta_{\theta,\mathcal{A}} f, \nabla^{s-1} \Delta_{\theta,\mathcal{A}} f\rangle d\nu\\
&& +2  \int_{\Sph^2} \langle [\nabla_{\nabla^\perp f}, \nabla^{s-1}] \Delta_{\theta,\mathcal{A}} f, \nabla^{s-1} \Delta_{\theta,\mathcal{A}} f\rangle d\nu\\
&=& 2  \int_{\Sph^2} \langle [\nabla_{\nabla^\perp f}, \nabla^{s-1}] \Delta_{\theta,\mathcal{A}} f, \nabla^{s-1} \Delta_{\theta,\mathcal{A}} f\rangle d\nu.
\end{eqnarray*}
where $[ \;,\;]$ denotes the commutator. For the last term of the above equation we use the estimate
$$\| [ \nabla^l ,\nabla_u] g \| = \nabla^l (ug) - u \nabla^l g\|_{H^0} \leq K \big(\| u\|_{H^l} \|g\|_{C^0} + \|\nabla u \|_{C^0} \|g\|_{H^{l-1}}\big),
$$ 
which can be found in \cite[Ch.~13, Prop.~3.7]{taylor}. 
We let $l=s-1,$ $u = \nabla^\perp f$, and $ g = \Delta_{\theta,\mathcal{A}} f$ from which we obtain
$$
\int_{\Sph^2} \big\langle[\nabla_{S_{\theta} f}, \nabla^{s-1}]\Delta_{\theta,\mathcal{A}} f,
\nabla^{s-1} \Delta_{\theta,\mathcal{A}} f\big\rangle \mu \leq K \|\nabla^\perp f\|_{C^1} \| \Delta_{\theta,\mathcal{A}} f \|^2_{H^{s-1}}.
$$
Thus,
$$
\frac{d}{dt} \|\Delta_{\theta,\mathcal{A}} f \|^2_{H^{s-1}} \leq K \|\nabla^\perp f\|_{C^1} \| \Delta_{\theta,\mathcal{A}} f \|^2_{H^{s-1}}.
$$
If $\|\nabla^\perp \|_{C^1} \|$ is finite then, using Gronwall's inequality, we get
\[
\|\Delta_{\theta,\mathcal{A}} f \|^2_{H^{s-1}} \leq \exp(K \int_0^t\|\nabla^\perp f(s)\|_{C^1}ds ).
\]
However, we already proved that $f$ belongs to $C^{2+\alpha}(\Sph^2,\R)$. 
The same reasoning combined with a time reversal gives boundedness on $(-T_b,0].$ This completes the proof.
\end{proof}

%
%
\subsection{Central extension and  the second advected quantity}\label{sec:central extension}
To obtain the full equations \eqref{eq-SWQG} as geodesics, we impose the term generated by the factor $(\frac{2}{{\rm Ro}} + 2 h)z$ into the geometric formulation \eqref{Geo-eq-S2} by considering the $1$-dimensional central extension of $\cD_q(\Sph^3)$.

Following \cite{Ebin-Preston-arXiv, Lee-Pre, Suri-2023, Tauchi-Yoneda-2} we consider the central extension of the Lie algebra $\cg=T_e\cD_{q}(\Sph^3)$ with $\mathbb{R}$, which we denoted by $\hat\cg=\cg\rtimes_\Omega\mathbb{R}$. 
For $u=\Ss_\theta f$ and $v=\Ss_\theta  g$ in $\cg$, consider the trivial cocycle  
\begin{equation}\label{Cocycle Omega on S3}
\Omega(u,v)=\int_{\Sph^3}\varphi\{f,g\}d\mu
\end{equation}
where $\varphi:\Sph^3\longrightarrow\mathbb{R}$ is a fixed function. The contribution induced by $\Omega$ in the geodesic equation accounts for that part of the Coriolis force induced by   $(\frac{2}{{\rm Ro}} + 2 h)z$ in \eqref{eq-SWQG}. 
We recall that for $(u,a),(v,b)\in \cg\rtimes_\Omega\mathbb{R}$, the Lie bracket is
\begin{equation*}
\big[  (u,a),(v,b)\big]=\big(  \Ss_\theta\{f,g\}  ,  \Omega(u,v)  \big).
\end{equation*}
By construction, the element with the form $(0,a)\in\hat{\cg}$ belongs to the center of the Lie algebra $\hat{\cg}$.
The inner product on $\hat\cg$ is given by
\begin{equation}\label{metric hat S3}
\llangle  (u,a),(v,b)  \rrangle^{\mathcal{\A}} =  \langle u , v \rangle^{\mathcal{\A}} + ab.
\end{equation}
Now, we will find an  operator  $T:\cg\longrightarrow \cg$ defined by the relation 
\[
\langle Tu,v \rangle^{\mathcal{\A}}=\Omega(u,v).
\]
%
%
\begin{Lem}
For any $u=\Ss_\theta f,v=\Ss_\theta g\in\cg$ the operator 
\[
T:\cg\longrightarrow \cg\quad;\quad u\longmapsto \Ss_{\theta}(\Delta_{\theta,\mathcal{A}})^{-1}\{\varphi,f\}
\]
satisfies $\Omega(u,v)=\langle Tu,v\rangle_{\mathcal{A}}$.
\end{Lem}
\begin{proof}   
Using the fact that 
\[
\int_{\Sph^3}\varphi\{f,g\}d\mu = \int_{\Sph^3} \varphi (\Ss_\theta f)(g)  d\mu = -\int_{\Sph^3} g (\Ss_\theta f)(\varphi)  d\mu 
= \int_{\Sph^3}  \{\varphi,f\} g d\mu 
\]
we have
\begin{eqnarray*}
\Omega(u,v)   &=&   \int_{\Sph^3}  \{\varphi,f\} g d\mu  \\
&=&   \int_{\Sph^3} \Big(\Delta_{\theta,\mathcal{A}}  \Delta_{\theta,\mathcal{A}}^{-1}\{\varphi,f\}\Big)  g  d\mu\\
&=&   \int_{\Sph^3}  g^{\Sph^3} \big( \mathcal{A}\Ss_\theta\Delta_{\theta,\mathcal{A}}^{-1}\{\varphi,f\}~,~\Ss_\theta g \big) d\mu\\
&=&  \langle \Ss_\theta\Delta_{\theta,\mathcal{A}}^{-1}\{\varphi,f\} , \Ss_\theta g\rangle_{\mathcal{A}}
\end{eqnarray*}
which completes the proof.
\end{proof}
%
%


%
%
\subsection*{The geodesic equation on $\widehat{G}=\widehat{\D_{q} (\Sph^3})$.   } 
Suppose the Lie group $\widehat{G}$ exists and is the 1-dimensional central extension of ${\D_{q} (\Sph^3})$, so that $T_e\widehat{G}=\hat\cg=\cg\rtimes_{\Omega}\R$.  
Now we extract the Euler-Arnold (geodesic) equation of the group $\widehat{G}$ with the (weak) metric \eqref{metric hat S3}. 
For the coadjoint operator on $\hat\cg$, using \cite[Sec.~2.1]{Suri-2023} we have
\[
\widehat{\ad}_{(u,a)}^*:\hat{\cg}\longrightarrow \hat{\cg}; \quad (v,b)\longmapsto (\ad^*_uv  - b Tu,0).
\]
Moreover, for the curve $(u,a):(-\epsilon,\epsilon)\longrightarrow \hat{\cg}$  the corresponding Euler-Arnold  equation on $(\widehat{G},\llangle  \;,\;  \rrangle)$ is
\begin{equation}\label{E-A-Eq hat G}
\left\{ \begin{array}{ll}  \partial_tu  +  \ad^*_uu  -  a(t)Tu  =  0   \\
\partial_ta(t)=0 .
\end{array}\right.
\end{equation}
The second equation implies that $a(t)=a$ is constant. However, for arbitrary  $u=\Ss_\theta f,v=\Ss_\theta g,w =\Ss_\theta h\in\cg$ we have
\begin{eqnarray*}
\langle \ad^*_uv,w\rangle^\mathcal{A}   &=& \langle v , \ad_uw\rangle^\mathcal{A}  \\
&=&  - \int_{\Sph^3} g^{\Sph^3} \big( \mathcal{A}\Ss_\theta g , \Ss_\theta\{f,h\} \big) d\mu  \\
&=&  - \int_{\Sph^3}  ( \Delta_{\theta,\mathcal{A}} g) \{f,h\}  d\mu  \\
&=&   \int_{\Sph^3}  \{f , \Delta_{\theta,\mathcal{A}} g\} h  d\mu  \\
&=&   \int_{\Sph^3} \Big(\Delta_{\theta,\mathcal{A}}  \Delta_{\theta,\mathcal{A}}^{-1}\{f , \Delta_{\theta,\mathcal{A}} g\}\Big)  h  d\mu\\
&=&   \int_{\Sph^3}  g^{\Sph^3} \big( \mathcal{A}\Ss_\theta\Delta_{\theta,\mathcal{A}}^{-1}\{f , \Delta_{\theta,\mathcal{A}} g\}~,~\Ss_\theta h \big) d\mu\\
&=&  \langle \Ss_\theta\Delta_{\theta,\mathcal{A}}^{-1}\{f , \Delta_{\theta,\mathcal{A}} g\} , \Ss_\theta h\rangle_{\mathcal{A}}.
\end{eqnarray*}
As a result, the Euler-Arnold equation \eqref{E-A-Eq hat G} takes the form
\begin{eqnarray*}
0&=& \partial_t\Ss_\theta f +     \Ss_\theta\Delta_{\theta,\mathcal{A}}^{-1}\{f , \Delta_{\theta,\mathcal{A}} f\} -a \Ss_\theta\Delta_{\theta,\mathcal{A}}^{-1}\{\varphi,f\} \\
&=& \Ss_\theta \Delta_{\theta,\mathcal{A}}^{-1} \Big(  \partial_t\Delta_{\theta,\mathcal{A}}f + \{f , \Delta_{\theta,\mathcal{A}} f\} -a \{\varphi,f\} \Big) 
\end{eqnarray*}
or equivalently
\begin{equation}\label{eq-SWQG-S3}
\partial_t\Delta_{\theta,\mathcal{A}}f + \{f , \Delta_{\theta,\mathcal{A}} f\} -a \{\varphi,f\}=0.
\end{equation}
\begin{Rem}\label{Rem phi}
Suppose that $\varphi$ is $E$-invariant. More precisely, for a given differentiable function $h:\Sph^2\longrightarrow \R$, we consider the lift of the map
\begin{equation}\label{Coriolis parameter}
\varphi:\Sph^2\longrightarrow \R\quad;\quad (z,\lambda)\longmapsto \frac{2z}{{\rm Ro}} + 2z h(z,\lambda)
\end{equation}
and  $\rho^2 = \gamma z^2$ where $\gamma$ and $\rm Ro$ are as described in \eqref{eq-SWQG}. 
Denote the lift of this function to $\Sph^3$ by $\varphi$. Then, for $a=1$, equation \eqref{eq-SWQG-S3} reduces to the equation on $\Sph^2$ given by 
\begin{equation}\label{eq-GQG-S2}
\partial_t(\gamma z^2-\Delta)f + \{f \:,\: (\gamma z^2-\Delta) f + \frac{2z}{{\rm Ro}} + 2z h \}=0 .
\end{equation}

For this Euler-Arnold equation, the transported quantity is given by $\hat{m}(t) = (\gamma z^2 - \Delta) f + \frac{2z}{{\rm Ro}} + 2z h$. Moreover, any uniform bound on $\hat{m}$ provides a uniform bound for the quantity $m$ in Theorem \eqref{Global existence}, and vice versa. Consequently, the global existence theorem applies in a similar manner to $\hat{m}$ as well.
\end{Rem}

%
%
\section{Curvature and Misio{\l }ek's criterion}\label{sec:curvature}
Let  $G$ be a manifold (possibly infinite-dimensional in the Banach category) which is also a topological group such that right translation is smooth (sometimes called a \emph{half-Lie group}). 
Moreover, suppose that $G$ is equipped with a  right-invariant (weak) metric defined by an inner product $\llangle \;,\;\rrangle$ on $\cg$. 
Let $[\;,\;]$ denote the Lie bracket on $\cg:=T_eG$. 
The \emph{Misio{\l }ek criterion} (also called \emph{Misio{\l }ek curvature}, cf.~\cite{Tauchi-Yoneda-2}) is given by
\begin{equation}\label{eq:MC_general}
\begin{aligned}
    {MC}\big( u,v \big)   
    &=  - \llangle  [u,v] , [u,v]\rrangle   -   \llangle  [[u,v],  v ],u \rrangle\\
    &=   \llangle   ad_u^*[u,v]    +       ad_{[u,v]}^*u   ,  v\rrangle .
\end{aligned}
\end{equation}
Its significance is the following: if $u$ is a stationary solution of the Euler--Arnold equation associated with $(G,\llangle \;,\;\rrangle)$ and there exists $v\in \cg$ such that $MC(u,v)>0$, then the sectional curvature in the plane spanned by $u$ and $v$ is positive and there exists a conjugate point for the geodesic curve corresponding to $u$.
For details, see \cite{Misiolek96,Tauchi-Yoneda-2}.

Suppose $M$ is a closed Riemannian manifold with volume form $\vol$. 
Then, for $G=\cD^s_{\vol}(M)$ equipped with the right-invariant $L^2$ metric and $u,v\in T_e\cD^s_{\vol}(M)$, the formula \eqref{eq:MC_general} reduces to the original criterion by Misio{\l }ek~\cite{Misiolek96}
\begin{equation*}
MC_{\vol}(u,v):=  \langle  \nabla_{u}[u,v] + \nabla_{[u,v]}u , v \rangle_{L^2} .
\end{equation*}

In the case where \( M \) is a 2-dimensional manifold and \( u = \nabla^\perp f \), \( v = \nabla^\perp g \) for smooth functions \( f, g \in C^\infty(M) \), the criterion simplifies to
\begin{equation}\label{MC original}
MC_{\vol}(u,v)=\langle \Delta \{f,g\}  ,  \{f,g\}  \rangle_{L^2}   - \langle  \{ \Delta f , g\}  ,  \{f,g\}  \rangle_{L^2}.    
\end{equation}
In the next Lemma, we compute Misio{\l }ek's criterion for $\big(\cD^s_q(\Sph^3)\:,\: \langle\:,\:\rangle_{\mathcal{A}}\big)$ and its reduced version on $\cD^s_\nu(\Sph^2)$.
%
%
%
\begin{Lem}\label{lemma MC on Dq}
For   $u=S_\theta f$ and $v=S_\theta g$ we have the following results. \\
\textbf{(a)} Misio{\l }ek's criterion is given by
\begin{eqnarray}\label{eqn MC S3}
{MC}^\mathcal{A}\big( u,v\big) &=&   - \langle \Delta_{\theta,\mathcal{A}} \{f,g\}  ,  \{f,g\}  \rangle_{L^2}   -  \langle  \{g,\Delta_{\theta,\mathcal{A}} f\}  ,  \{f,g\}  \rangle_{L^2} . 
\end{eqnarray}
\textbf{(b)} Using the projection $\Pi$ introduced in \eqref{Projection}, the above criterion reduces to
\begin{eqnarray}\label{eqn Misiolek Curvature on S2}
{MC}^{\mathcal{A}}\big( u,v \big)  &=&  \gamma\langle f\{z^2,g\},\{f,g\}\rangle_{L^2} + MC_\nu(u,v).
\end{eqnarray}
\end{Lem}

\begin{proof}
The proof of part \textbf{a} is similar to that of Lemma 2.4 in \cite{Suri-CPS}. For the second part, we observe that  
\begin{eqnarray*}
{MC}^\mathcal{A}\big( u,v\big) &=&   - \langle (\gamma z^2-\Delta) \{f,g\}  ,  \{f,g\}  \rangle_{L^2}   -  \langle  \{g,(\gamma z^2-\Delta) f\}  ,  \{f,g\}  \rangle_{L^2}\\
&=&   - \gamma\langle  z^2 \{f,g\}  ,  \{f,g\}  \rangle_{L^2}      +     \langle \Delta \{f,g\}  ,  \{f,g\}  \rangle_{L^2}    \\
&&  +  \gamma \langle  \{ z^2 f , g\}  ,  \{f,g\}  \rangle_{L^2}   - \langle  \{ \Delta f , g\}  ,  \{f,g\}  \rangle_{L^2}\\
&=&   - \gamma\langle  z^2 \{f,g\}  ,  \{f,g\}  \rangle_{L^2}      +     \langle \Delta \{f,g\}  ,  \{f,g\}  \rangle_{L^2}    \\
&&  +  \gamma \langle  z^2\{  f , g\}  ,  \{f,g\}  \rangle_{L^2}   +  \gamma \langle  f\{ z^2  , g\}  ,  \{f,g\}  \rangle_{L^2} \\ && - \langle  \{ \Delta f , g\}  ,  \{f,g\}  \rangle_{L^2}\\
&=&  \gamma \langle  f\{ z^2  , g\}  ,  \{f,g\}  \rangle_{L^2}      +     \langle \Delta \{f,g\}  ,  \{f,g\}  \rangle_{L^2} \\ &&  - \langle  \{ \Delta f , g\}  ,  \{f,g\}  \rangle_{L^2}\\
&=&  \gamma \langle  f\{ z^2  , g\}  ,  \{f,g\}  \rangle_{L^2}      +    MC_\nu(u,v),
\end{eqnarray*}
where the last equality follows from the formula~\eqref{MC original}.
\end{proof}

We shall now use the result in Lemma~\ref{lemma MC on Dq} to investigate conjugate points near the \emph{trade-wind current}, which is the stationary solution with vector field $u_T=\nabla^\perp T$ for the stream function $T(z):=\frac{1}{2}\sqrt{\frac{15}{8\pi}} z^2$.
In particular, along the trade-wind solution we have from Lemma~\ref{lemma MC on Dq} that for any $v = \nabla^\perp g$
\begin{eqnarray*}
    {MC}^\mathcal{A}\big( u_T,v\big) &=&     \gamma \int_{\Sph^2}  |z\{ z^2  , g\}|^2   d\mu      +    MC_\nu(u_T,v).
\end{eqnarray*}

\begin{cor}\label{Cor-TW}
If the Lamb parameter fulfills $\gamma > \frac{-MC_\nu(u_T, v)}{\int_{\Sph^2} |z\{z^2, g\}|^2  d\nu}$, then conjugate points along the trade-wind current appear. 
Comparing this with the fact that the curvature of $u_T$, in the absence of the Coriolis force, is non-positive in most directions \cite{Luk, Suri-2023}, we observe that for a suitable choice of the parameter $\gamma$, the Misio{\l }ek criterion $MC^{\mathcal{A}}(u_T,v)$  (and consequently the sectional curvature) in any non-zonal direction $v = \nabla^\perp g$ is positive.
\end{cor}

\paperonly{
\textbf{Question 1} Can we visualize this mathematical observation?
{\color{red}KM: I'm working with a visualization, but I have some bug that I need to fix. I suggest we submit to arXiv without the numerics.}
}
\begin{example}\label{Examp}
For the stream function  
\[
g(z,\lambda)=P^2_3(z)\cos(2\lambda)=15z(1-z^2)\cos(2\lambda)
\]
we obtain
\begin{eqnarray*}
&& \{z^2,g(z)\} =-60z^2(1-z^2)\sin(2\lambda)\\
&& \Delta\{z^2,g(z)\}=-60(2-22z^2+20z^4)\sin(2\lambda)\\
&&\{\Delta z^2,g(z)\} =-60 (-6) z^2(1-z^2)\sin(2\lambda)\\
&&MC(u_T,u_g)=60^2\pi\int_{-1}^{+1}(z^2-z^4)(2-22z^2+20z^4) +6z^4(1-z^2)^2 dz = -60^2\pi 0.4063\\
&&\int_{\Sph^2}z^2\{z^2,g\}^2d\nu=60^2\pi 0.02309
\end{eqnarray*}
As a result, for 
\begin{equation*}
    \gamma>\frac{-MC_\nu(u_T, u_g)}{\int_{\Sph^2} |z\{z^2, g\}|^2  d\nu} =\frac{0.4063}{0.02309} =17.6
\end{equation*}
conjugate points along the geodesic generated by the trade-wind current appear if we perturb the initial velocity field by $u_g$.
\end{example}

Our developments concerning curvature and conjugate points has so far been without the Coriolis force.
We now give a corresponding result on the central extension group in section~\ref{sec:central extension}, thus allowing non-vanishing Coriolis forces.

\begin{prop} For $(u,a),(v,b)\in \hat\cg$ and $\varphi$ as in Remark \ref{Rem phi}, we have
\begin{equation}
\widehat{MC}^\mathcal{A}\big( (u,a),(v,b) \big) =  {MC}^\mathcal{A}\big( u,v\big) - (\langle\{\varphi,f\} , g\rangle)^2 
-a \langle\{\varphi,\{f,g\}\} , g\rangle.
\end{equation}
\end{prop}
\begin{proof}
Set $(w,d):=\big[(u,a),(v,b)\big].$ Then we get $ w= [u,v]$ and  $d=\Omega(u,v)=\langle \{\varphi,f\},g\}$. As a result, 
\begin{align*}
 \widehat{MC}^\mathcal{A}\big( (u,a),(v,b) \big) &= \llangle \widehat{\ad}^*_{(u,a)}(w,d) + \widehat{\ad}^*_{(w,d)}(u,a) ,(v,b)\rrangle^\mathcal{A}\\
& = \llangle (\ad^*_{u}w + \ad^*_wu -dTu-aTw ,0) ,(v,b)\rrangle^\mathcal{A}\\
& = \langle \ad^*_{u}w + \ad^*_wu -dTu-aTw  ,v\rangle^\mathcal{A}\\
& = \langle \ad^*_{u}w + \ad^*_wu ,v\rangle^\mathcal{A}  -d \langle Tu,v\rangle^\mathcal{A} -a \langle Tw  ,v\rangle^\mathcal{A}\\
& = {MC}^\mathcal{A}\big( u,v\big)  -d \langle \{\varphi,f\},g\rangle -a \langle \{\varphi,\{f,g\}\},g  \rangle\\
& = {MC}^\mathcal{A}\big( u,v\big)  -( \langle \{\varphi,f\},g\rangle)^2 -a \langle \{\varphi,\{f,g\}\},g \rangle .
\end{align*}
\end{proof}

In particular, for the trade-wind solution we have the following result.

\begin{cor}\label{Cor-Tw-Coriolis}
Let $\varphi$  be given by equation \eqref{Coriolis parameter} with $h=h(z)$. Then, the steam function $T(z)$ of the trade wind velocity field $u_T$  is a solution of the equation \eqref{eq-GQG-S2}. Using Corollary \ref{Cor-TW} and the previous proposition we get
\begin{eqnarray*}
\widehat{MC}^\mathcal{A}\big( (u_T,a),(v,b) \big) &=& \gamma \int_{\Sph^2}  |z\{ z^2  , g\}|^2   d\mu      +    MC_\nu(u_T,v) 
\\
&&- a \langle \{\varphi,\{T(z),g\}\},g \rangle.
\end{eqnarray*}
Notably, for suitable choices of $\gamma$ and $a$, we can always get a positive value of $\widehat{MC}^\mathcal{A}\big( (u_T,a),(v,b) \big)$. 
\end{cor}
\begin{Rem} In light of Corollaries \ref{Cor-TW}, \ref{Cor-Tw-Coriolis}, and Example \ref{Examp}, we observe the stabilizing effect of the Lamb parameter $\gamma$  in our specific cases. It would be interesting to further explore the influence of the parameters $\gamma$ and $a$ on the behavior of  other zonal solutions, as was done numerically in \cite{FrCaCiGe2024,Luesink-2024,Franken-2024}, and for the quasi-geostrophic setting in \cite{Lee-Pre}. Additionally, the term $h$ in \eqref{eq-SWQG}, which encodes the topography, may also play a significant role in determining the curvature and the appearance of conjugate points. 
\paperonly{
{\color{blue} Ali: 
Here, we highlight some possible directions for further research based on our approach, which also connects this work more concretely to the ongoing developments in the field.}
}
\end{Rem}

%
%
\bibliographystyle{alpha}
\bibliography{qg_refs}

\end{document}